\documentclass[12pt]{article}

\usepackage[a4paper,left=2cm,right=2cm,top=2cm,bottom=2cm]{geometry}

\usepackage{graphicx}
\graphicspath{{img/}}

\usepackage{natbib}
\setcitestyle{round,comma}%,notesep={ }}
\defcitealias{OEIS}{OEIS}
\newcommand{\oeisurl}[1]{\href{https://oeis.org/#1}{#1}}
\newcommand{\oeis}[1]{\citetalias{OEIS} \oeisurl{#1}}
\makeatletter
\renewcommand*{\NAT@nmfmt}[1]{\textsc{#1}}
\makeatother

\usepackage{amsmath,amssymb}
\usepackage{amsthm}
\newtheorem{theorem}{Theorem}
\newtheorem{proposition}{Proposition}
\theoremstyle{definition}
\newtheorem{definition}{Definition}

\usepackage{enumitem}
\setlist{nosep}

\newcommand{\gen}[1]{\left\langle #1 \right\rangle}
\newcommand{\subl}[1]{\mathcal{L}(#1)}
\newcommand{\simp}[1]{\mathcal{S}(#1)}
\newcommand{\ori}[1]{\left\lfloor #1 \right\rceil}
\newcommand{\latt}[1]{\mathcal{L}\!\gen{#1}}

\usepackage{doi}
\usepackage{hyperref}

\begin{document}

\title{Coweight lattice $A^*_n$ and lattice simplices}
\author{Andrey Zabolotskiy \\
{\small \href{mailto:zabolotskiy@phystech.edu}{zabolotskiy@phystech.edu}} \\
{\small Dukhov Automatics Research Institute (VNIIA), 127055 Moscow, Russia}
}
\date{}
\maketitle

\begin{abstract}
There exist as many index-$k$ sublattices of the hexagonal lattice up to isometry as there exist lattice triangles with normalized volume $k$ up to unimodular equivalence, which can be explained using orbifolds. In dimension 3, it was noted that the number of sublattices of the fcc and the bcc lattices and the number of lattice tetrahedra all seem to be the same.
We provide a bijection between the sublattices of the coweight lattice $A^*_n$ and the $n$-dimensional lattice simplices.
It explains, proves, and generalizes the observed coincidences to arbitrary dimension.

\end{abstract}

Sections \ref{sec:subl} and \ref{sec:simp} are introductory: the needed concepts known from literature and our notation for them are introduced. Some coincidences between the number of sublattices and the number of lattice simplices are noted. In Section \ref{sec:bij}, the result is stated and proven: using only elementary methods, we provide a bijection between the sublattices of the $A^*_n$ lattice and the $n$-dimensional lattice simplices that preserves certain equivalence relations, which ultimately explains the said coincidences.

\section{Sublattices}
\label{sec:subl}

The full-rank sublattices of the $n$-dimensional root lattice $A_n$ and its dual, the coweight lattice $A^*_n$, have been studied and counted in different contexts. \cite{SloaneHex} and \cite{Rutherford4} have counted the sublattices of the hexagonal lattice $A_2$ (or $A^*_2$) of any given index. \cite{HananyCountingOrbifolds,HananySublattice,HananyAbelian} have drawn connections between abelian orbifolds, simplical toric diagrams, brane tilings, and the sublattices of certain lattices (in particular, noting that the sublattices of the hexagonal lattice up to isometry and the lattice triangles up to unimodular equivalence both correspond to abelian orbifolds of $\mathbb{C}^3$). \cite{HartForcade} have counted the sublattices of the fcc lattice $A_3$ and the bcc lattice $A^*_3$, obtaining matching numbers. \cite{RRSubl} have related the sublattices of $A_n$ to the discrete Laplacians of finite graphs. Also, \cite{Montagard} studied the connection between regular lattice simplices and the root system $A_n$. An introduction to the lattices $A_n$ and $A^*_n$ is given, e.~g., in \cite{book:perfect}, Sec.~4.2.

A sublattice can be defined by its basis. Columns of coefficients of the basis vectors with respect to the basis of the parent $n$-dimensional lattice can be assembled into an integer $n\times n$ matrix $B$. We denote the sublattice (of some lattice, depending on the context) generated by the basis $B$ as $\subl{B}$. We allow negative orientation of the basis, so the index of the sublattice is given by $|\det B|$. The choice of the basis of a sublattice is apparently ambiguous. The unique representative of the class of basis matrices generating a particular sublattice having a certain canonical form is known as the Hermite normal form (we will not need it though).

Most of the above-mentioned papers deal with the following equivalence relation for sublattices.
\begin{definition}\label{def:s-eq}
We say that two sublattices $\subl{B_1}$ and $\subl{B_2}$ of some parent lattice $\mathcal{L}$ are \emph{isometric}, denoted $\subl{B_1}\sim \subl{B_2}$, if one transforms into another via an isometric automorphism of $\mathcal{L}$.
\end{definition}
For example, the two sublattices of $A_2$ shown in Fig.~\ref{fig:equiv2-s} are isometric because they are related via a $\pi/3$ rotation (or via a certain reflection which is also an isometry of $A_2$).

\begin{figure}
\centering
\includegraphics[width=0.3\columnwidth]{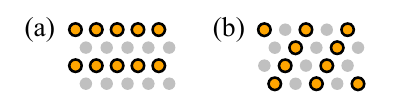}
\caption{Two index-2 sublattices of the hexagonal affine lattice $A_2$, isometric to each other.}
\label{fig:equiv2-s}
\end{figure}

We denote the number of equivalence classes of index-$k$ sublattices of $A_n$ with respect to isometricity by $\beta_{n,k}$ (the isometries respect the index).

The sequence $\{\beta_{2,k}\}_{k}$ is \oeis{A003051}; that entry provides several explicit formulas for this sequence based on the above-referenced papers. The sequence $\{\beta_{3,k}\}_{k}$ is \oeis{A159842}; the explicit formula for this sequence had been derived by \cite{HananySublattice} using Polya's enumeration theorem (it is given for ``the tetrahedral lattice'', which is more commonly known as the diamond crystal structure, consisting of two interpenetrating copies of the $A_3$ lattice; only one copy can contain a sublattice).

There is another commonly used equivalence relation for sublattices.
\begin{definition}\label{def:s-eq-e}
We say that two sublattices $\subl{B_1}$ and $\subl{B_2}$ of some parent lattice $\mathcal{L}$ are \emph{properly isometric}, denoted $\subl{B_1}\sim^+ \subl{B_2}$, if one transforms into another via an orientation-preserving isometric automorphism of $\mathcal{L}$.
\end{definition}
That is, no reflections are allowed in this case. The two sublattices in Fig.~\ref{fig:equiv2-s} are properly isometric, too. We denote the number of index-$k$ sublattices of $A_n$ inequivalent with respect to proper isometricity by $\beta^+_{n,k}$. The sequence $\{\beta^+_{2,k}\}_k$ is \oeis{A145394}.

If we consider affine lattices, their affine sublattices, and their isometric affine automorphisms (i.~e., including translations) instead, the equivalence classes of sublattices defined by properly adjusted definitions \ref{def:s-eq} and \ref{def:s-eq-e} remain the same.
It is possible to consider an equivalence relation that identifies the sublattices related via an arbitrary isometry of the ambient vector space, that is, via non-crystallographic rotations (and maybe reflections), but the problem of counting such equivalence classes is more difficult. It is related to the study of coincidence site lattices, which were counted in $A_4$ by \cite{CSL4D,A4CSL}.

\section{Lattice simplices}
\label{sec:simp}

The studies of the lattice polytopes is an established field [\cite{book:Haase}]. Lattice polytopes are commonly studied up to the unimodular equivalence: two lattice polytopes $\mathcal{S}_1$ and $\mathcal{S}_2$ (or any other geometric entities defined on a lattice) are \emph{unimodularly equivalent}, denoted $\mathcal{S}_1 \cong \mathcal{S}_2$, if they are related via an affine transformation of the ambient space preserving the parent lattice. Lattice polytopes are closely related to toric varieties [\cite{book:toric}].

We will only consider $n$-dimensional lattice simplices, which can be identified with unordered collections of $n+1$ vertices. The particular type of the parent lattice does not matter because no isometries are involved. Without loss of generality, we will always assume that one of the vertices is at the origin. The columns of the coordinates of the other $n$ vertices with respect to the basis of the parent lattice form the integer non-degenerate matrix $T$. We denote the ordered lattice simplex (i.~e., the $(n+1)$-tuple of vertices), which is the convex hull of the vectors given by $T$ together with the origin, by $\simp{T}$, the corresponding oriented lattice simplex (i.~e., the lattice simplex ordered modulo even permutations) as $\ori{\simp{T}}$, and the (unordered, unoriented) lattice simplex formed by these vectors by $|\simp{T}|$.
The volume of that simplex (induced by the parent lattice) is $\frac{1}{n!}|\det T|$; in the terms of \cite{book:Haase}, it has normalized volume (which is the volume induced by the parent lattice times $n!$) $|\det T|$. The unimodular equivalence preserves the volume.

The notion of the unimodular equivalence is applicable to ordered and oriented simplices as well as the unordered, unoriented ones. $\simp{T_1}\cong\simp{T_2} \Rightarrow \ori{\simp{T_1}}\cong\ori{\simp{T_2}} \Rightarrow |\simp{T_1}|\cong|\simp{T_2}|$ for all $T_{1,2}$, but the converse is not true. This motivates the following equivalence relations.
\begin{definition}\label{def:t-eq}
We say that two ordered lattice simplices, $\simp{T_1}$ and $\simp{T_2}$, are \emph{unimodularly equivalent as unordered, unoriented lattice simplices}, denoted $\simp{T_1}\sim \simp{T_2}$, if $|\simp{T_1}|\cong|\simp{T_2}|$.
\end{definition}
For example, the two lattice triangles shown in Fig.~\ref{fig:equiv2-t} are unimodularly equivalent as unordered, unoriented lattice simplices (in fact, the order of their vertices or their orientation is not even chosen and shown).

\begin{figure}
\centering
\includegraphics[width=0.3\columnwidth]{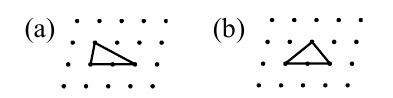}
\caption{Two unimodularly equivalent unordered, unoriented lattice triangles with lattice volume 2.}
\label{fig:equiv2-t}
\end{figure}

This equivalence relation respects the volume. We denote the number of equivalence classes of $n$-dimensional ordered lattice simplices with normalized volume $k$ with respect to this equivalence relation by $\tau_{n,k}$.

This equivalence relation is the same as the one introduced in \cite{HananyCountingOrbifolds} for toric diagrams based on the barycentric coordinates and the ``topological character''.
The sequence $\{\tau_{3,k}\}_k$, which counts lattice tetrahedra, was first computed by J.-O.~Moussafir, and \cite{HartForcade} noted the coincidence with $\{\beta_{3,k}\}_k$.
The sequences $\{\tau_{n,k}\}_k$ for $n=4,5,6$ are \citetalias{OEIS} \oeisurl{A173824}, \oeisurl{A173877}, \oeisurl{A173878}; they were studied in \cite{HananyAbelian} and also computed by \cite{Balletti}.

\cite{book:Karpenkov} studies the lattice geometry, lattice trigonometry, and their relation to the continued fractions; toric geometry is also known to have connections with continued fractions and Hirzebruch--Jung continued fractions [\cite{book:toric}, \S10.2]. In Sec.~6.5 of \cite{book:Karpenkov}, the oriented lattice triangles are counted up to the unimodular equivalence (also called integer congruence). This inspires the following equivalence relation.
\begin{definition}\label{def:t-eq-e}
We say that two ordered lattice simplices, $\simp{T_1}$ and $\simp{T_2}$, are \emph{unimodularly equivalent as oriented lattice simplices}, denoted $\simp{T_1}\sim^+ \simp{T_2}$, if $\ori{\simp{T_1}}\cong\ori{\simp{T_2}}$.
\end{definition}
We denote the number of equivalence classes of $n$-dimensional ordered lattice simplices with normalized volume $k$ with respect to unimodularly equivalence as oriented lattice simplices by $\tau^+_{n,k}$. \cite{book:Karpenkov} gives $\tau^+_{2,k}$ for $k\in\{1,\ldots,20\}$. One can see that they coincide with the corresponding values of $\beta^+_{2,k}$.

\section{Bijection}
\label{sec:bij}

The above-reviewed connections and observations lead to the following conjecture.
\begin{proposition}
$\beta_{n,k}=\tau_{n,k}$ and $\beta^+_{n,k}=\tau^+_{n,k}$ for all $n,k$.
\label{prop:coinc}
\end{proposition}
We will prove it using only elementary methods.

We need the following theorem.
\begin{theorem}
There is a bijection between the ordered lattice simplices of lattice volume $k$ on an $n$-dimensional lattice and the bases of index-$k$ sublattices of the coweight lattice $A^*_n$. That bijection maps
\begin{itemize}
\item bases generating the same sublattices to unimodularly equivalent ordered lattice simplices and vice versa,
\item bases generating properly isometric sublattices to lattice simplices unimodularly equivalent as oriented lattice simplices and vice versa,
\item bases generating isometric sublattices to lattice simplices unimodularly equivalent as unordered, unoriented lattice simplices and vice versa.
\end{itemize}
\label{thm:bij}
\end{theorem}
\begin{proof}
It is known that $\subl{B_1}=\subl{B_2} \Leftrightarrow \exists L\in \mathrm{GL}_n(\mathbb{Z})\colon B_1=B_2L$, i.~e., a basis change is represented in coordinates by the right multiplication by a unimodular integer matrix.

The isometries of the parent lattice are represented by the left multiplication of the sublattice basis matrix by certain matrices from $\mathrm{GL}_n(\mathbb{Z})$. The group of isometric automorphisms of the $A^*_n$ coweight lattice is isomorphic to $\mathfrak{S}_{n+1}\times\mathbb{Z}_2$ where $\mathbb{Z}_2$ corresponds to the space inversion (which is already taken into account since we consider the basis change simultaneously with the isometries) and the full symmetric group $\mathfrak{S}_{n+1}$ corresponds to the reflections and rotations permuting the set of $n$ certain appropriately chosen minimal lattice vectors $\mathbf{e}_1,\ldots,\mathbf{e}_n$ together with the vector $\mathbf{e}_0=-(\mathbf{e}_1+\dots+\mathbf{e}_n)$, also minimal [\cite{book:perfect}]. In the basis of these minimal vectors, these transformations are represented by the set $\mathfrak{P}_n$ of matrices generated by the set of all $n\times n$ permutation matrices, $\Pi_n$, and the matrices $P_{n,i}$ for $i\in\{1,\ldots,n\}$, where their elements are $P_{n,i;j,i}=-1$ and $P_{n,i;j,\ell}=\delta_{j,\ell}$ for $j\in\{1,\ldots,n\}$, $\ell\in\{1,\ldots,n\}\setminus\{i\}$. So $\subl{B_1}\sim\subl{B_2} \Leftrightarrow \exists L\in \mathrm{GL}_n(\mathbb{Z}),R\in\mathfrak{P}_n\colon B_1=RB_2L$; $\subl{\cdot}$ hereafter denotes a sublattice of the $A^*_n$ lattice in the basis of the appropriate minimal vectors $\mathbf{e}_1,\ldots,\mathbf{e}_n$.

The orientation-preserving isometric automorphisms of $A^*_n$ are represented by the matrices from $\mathfrak{P}_n$ with determinant 1. Consequently, $\subl{B_1}\sim^+\subl{B_2} \Leftrightarrow \exists L\in \mathrm{GL}_n(\mathbb{Z}),R\in\mathfrak{P}_n\cap \mathrm{SL}_n(\mathbb{Z})\colon B_1=RB_2L$.

On the other hand, the unimodular equivalence of the ordered lattice simplices is equivalent to the existence of a unimodular matrix relating their coordinates: $\simp{T_1}\cong\simp{T_2} \Leftrightarrow \exists L\in \mathrm{GL}_n(\mathbb{Z})\colon T_1=LT_2$.

Reordering of the vertices of the simplex $|\simp{T}|$ includes the permutations of the $n$ non-origin vertices represented by the right multiplication of the matrix $T$ by the permutation matrices from $\Pi_n$ and translating the simplex so that one of its vertices, $\mathbf{t}_i$, moves to the origin and the former origin vertex takes its place in the tuple of vertices as the new vertex $-\mathbf{t}_i$. These translations are represented by the matrices $P_{n,i}^\mathrm{T}$ where $P_{n,i}$ is defined as above. The matrix group generated by $\Pi_n$ and $\{P_{n_i}^\mathrm{T}\}_i$ is $\mathfrak{P}_n^\mathrm{T}$. So $|\simp{T_1}|=|\simp{T_2}| \Leftrightarrow \exists R\in\mathfrak{P}_n^\mathrm{T}\colon T_1 = T_2R$, and $\simp{T_1}\sim\simp{T_2} \Leftrightarrow \exists L\in \mathrm{GL}_n(\mathbb{Z}), R\in\mathfrak{P}_n^\mathrm{T}\colon T_1 = LT_2R$.

Similarly, since $\mathfrak{P}_n^\mathrm{T}\cap \mathrm{SL}_n(\mathbb{Z})=(\mathfrak{P}_n\cap \mathrm{SL}_n(\mathbb{Z}))^\mathrm{T}$, for the unimodular equivalence as oriented lattice simplices we have: $\simp{T_1}\sim^+\simp{T_2} \Leftrightarrow \exists L\in \mathrm{GL}_n(\mathbb{Z}), R\in(\mathfrak{P}_n\cap \mathrm{SL}_n(\mathbb{Z}))^\mathrm{T}\colon T_1 = LT_2R$.

From the above-given coordinate representations of the equivalence relations of the sublattices and the lattice simplices, it follows that there is a bijection between the bases of the sublattices of $A^*_n$ and the ordered lattice polytopes preserving the equivalence relations defined on them. That bijection is given by the matrix transposition:
\begin{align*}
\subl{B_1}=\subl{B_2} & \Leftrightarrow \simp{B_1^\mathrm{T}}\cong\simp{B_2^\mathrm{T}} \\
\subl{B_1}\sim\subl{B_2} & \Leftrightarrow \simp{B_1^\mathrm{T}}\sim\simp{B_2^\mathrm{T}} \\
\subl{B_1}\sim^+\subl{B_2} & \Leftrightarrow \simp{B_1^\mathrm{T}}\sim^+\simp{B_2^\mathrm{T}}
\end{align*}
\end{proof}
The bijection of Theorem~\ref{thm:bij} in the case $n=2, k=6$ is illustrated in Fig.~\ref{fig:bij}.

\begin{figure}
\centering
\includegraphics[width=\columnwidth]{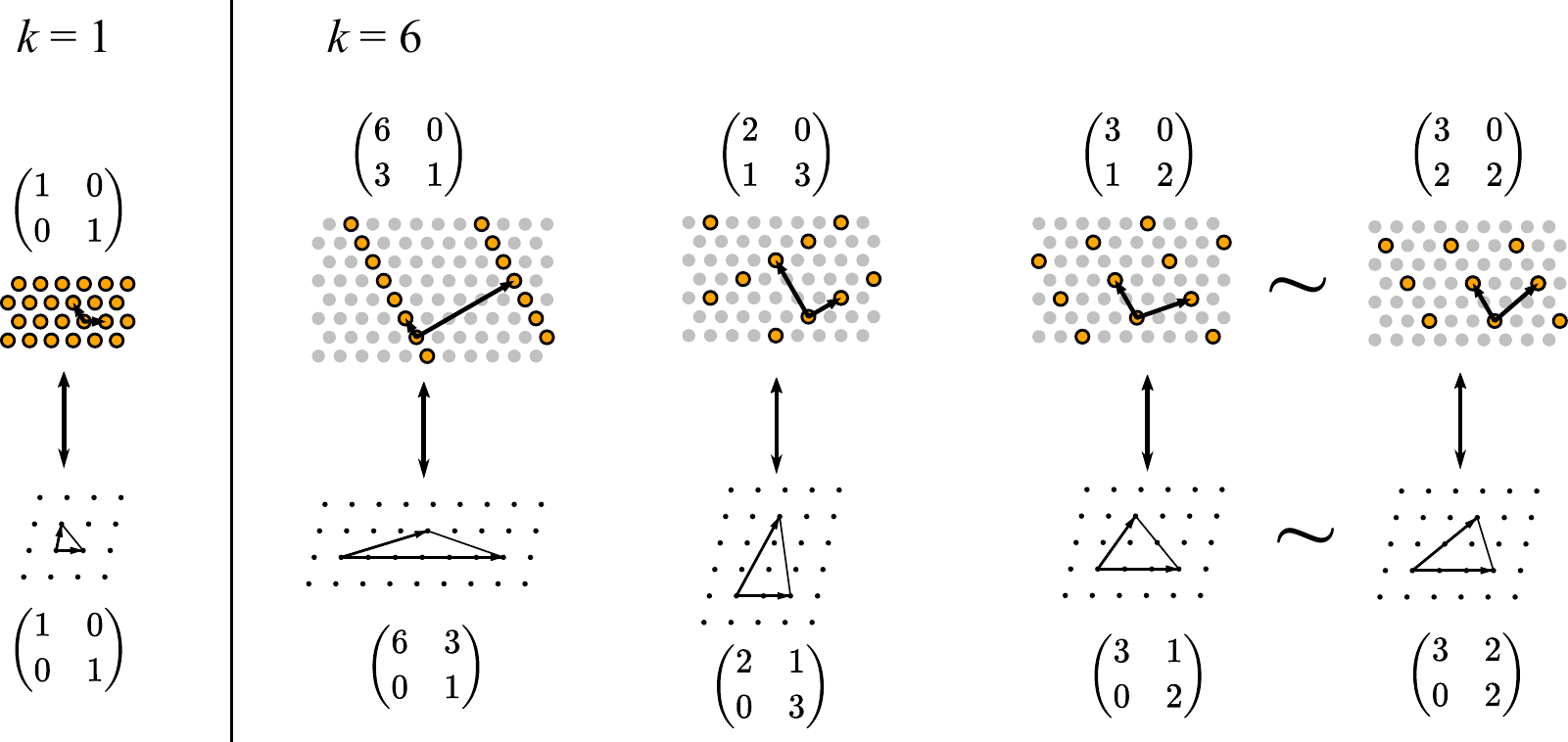}
\caption{
Top: the sublattices $\subl{B}$ of the hexagonal lattice $A^*_2$ (isomorphic to $A_2$), a pair of basis vectors generating each of them, and the matrix $B$ corresponding to that pair of vectors. On the left, the only index-1 sublattice of $A^*_2$ which is the parent lattice itself is shown; on the right, the 4 index-6 sublattices of $A^*_2$ are shown.
The sublattices are mutually non--properly isometric, and there is one pair of isometric sublattices.
Bottom: the lattice triangles $\simp{T}$ corresponding to these sublattices, a pair of vectors that form each of these triangles (their ordering is always in the positive direction), and the matrix $T=B^\mathrm{T}$ corresponding to that pair of vectors.
}
\label{fig:bij}
\end{figure}

To get Proposition~\ref{prop:coinc}, we now only need to prove the following proposition and apply it to the lattice $A_n$. 
\begin{proposition}
The number of mutually non-isometric or non--properly isometric index-$k$ sublattices is the same for a lattice $\mathcal{L}$ and its dual $\mathcal{L}^*$.
\label{prop:dualsubl}
\end{proposition}

There exists a canonical bijection between the sublattices of a lattice $\mathcal{L}$ and the superlattices of its dual $\mathcal{L}^*$, but it is of no use here. The automorphism groups of $\mathcal{L}$ and $\mathcal{L}^*$ are the same. So if we invoke Burnside's lemma to count the non-isometric sublattices (as is done by \cite{Rutherford4} and \cite{HananySublattice}), to prove Proposition~\ref{prop:dualsubl}, we only need to show that the number of sublattices of the given index fixed by any particular isometry is the same for $\mathcal{L}$ and $\mathcal{L}^*$. It would follow from the following proposition.
\begin{proposition}
For any isometric automorphism of a lattice $\mathcal{L}$, there exist bases in $\mathcal{L}$ and $\mathcal{L}^*$ such that they are unimodularly transformed by that isometric automorphism in the same way.
\end{proposition}
\begin{proof}
Let $\mathcal{L}=\latt{B}$ be an $n$-dimensional lattice generated by some basis vectors (written in the orthonormal basis of the ambient space $\mathbb{R}^n$) forming the matrix $B\in \mathrm{GL}_n(\mathbb{R})$ (we can assume that the lattice is unimodular without loss of generality). Then $\latt{B}^*=\latt{B^{-\mathrm{T}}}$ where $B^{-\mathrm{T}}\equiv (B^\mathrm{T})^{-1}$. An isometric automorphism of $\latt{B}$ is an orthogonal transformation of the coordinates $R_o\in \mathrm{O}_n(\mathbb{R})$ such that $\latt{R_oB}=\latt{B}$, which means that $R_oB=BR$ for some $R\in \mathrm{GL}_n(\mathbb{Z})$, and also $R_oB^{-\mathrm{T}}=B^{-\mathrm{T}}R'$, $R'\in \mathrm{GL}_n(\mathbb{Z})$. The statement of the proposition means that there is a basis $B'$ for $\mathcal{L}^*$ such that $R_oB'=B'R$. It holds if we take $B'=B^{-\mathrm{T}}L$ where $L= (R')^{-1}R \in \mathrm{GL}_n(\mathbb{Z})$.
\end{proof}

Proposition~\ref{prop:dualsubl} follows. Recalling Theorem~\ref{thm:bij}, Proposition~\ref{prop:coinc} follows.

\section*{Acknowledgements}
The author is grateful to Oleg Karpenkov for helpful advice.

\addcontentsline{toc}{section}{\refname}%
\bibliography{lattice}

\begin{thebibliography}{16}
\providecommand{\natexlab}[1]{#1}
\providecommand{\url}[1]{\texttt{#1}}
\expandafter\ifx\csname urlstyle\endcsname\relax
  \providecommand{\doi}[1]{doi: #1}\else
  \providecommand{\doi}{doi: \begingroup \urlstyle{rm}\Url}\fi

\bibitem[Amini and Manjunath(2010)]{RRSubl}
Omid Amini and Madhusudan Manjunath.
\newblock Riemann-{R}och for sub-lattices of the root lattice {$A_n$}.
\newblock \emph{El. J. Comb.}, 17:\penalty0 R124, 2010.
\newblock URL
  \url{https://www.combinatorics.org/ojs/index.php/eljc/article/view/v17i1r124}.

\bibitem[Baake and Zeiner(2008)]{CSL4D}
Michael Baake and Peter Zeiner.
\newblock Coincidences in four dimensions.
\newblock \emph{Phil. Mag.}, 88:\penalty0 2025--2032, 2008.
\newblock \doi{10.1080/14786430701846206}.

\bibitem[Balletti(2020)]{Balletti}
Gabriele Balletti.
\newblock Enumeration of lattice polytopes by their volume.
\newblock \emph{Discrete Comput. Geom.}, 2020.
\newblock \doi{10.1007/s00454-020-00187-y}.

\bibitem[Bernstein et~al.(1997)Bernstein, Sloane, and Wright]{SloaneHex}
M.~Bernstein, N.~J.~A. Sloane, and Paul~E. Wright.
\newblock On sublattices of the hexagonal lattice.
\newblock \emph{Discrete Math.}, 170:\penalty0 29--39, 1997.
\newblock \doi{10.1016/0012-365X(95)00354-Y}.

\bibitem[Cox et~al.(2011)Cox, Little, and Schenk]{book:toric}
David~A. Cox, John~B. Little, and Henry~K. Schenk.
\newblock \emph{Toric varieties}.
\newblock {AMS}, 2011.

\bibitem[Davey et~al.(2010)Davey, Hanany, and Seong]{HananyCountingOrbifolds}
John Davey, Amihay Hanany, and Rak-Kyeong Seong.
\newblock Counting orbifolds.
\newblock \emph{J. High Energ. Phys.}, 2010:\penalty0 10, 2010.
\newblock \doi{10.1007/JHEP06(2010)010}.

\bibitem[Haase et~al.(2012)Haase, Nill, and Paffenholz]{book:Haase}
Christian Haase, Benjamin Nill, and Andread Paffenholz.
\newblock Lecture notes on lattice polytopes.
\newblock Preliminary version, 2012.
\newblock URL
  \url{https://polymake.org/polytopes/paffenholz/data/preprints/ln\_lattice\_polytopes.pdf}.

\bibitem[Hanany and Seong(2011)]{HananyAbelian}
Amihay Hanany and Rak-Kyeong Seong.
\newblock Symmetries of abelian orbifolds.
\newblock \emph{J. High Energ. Phys.}, 2011:\penalty0 27, 2011.
\newblock \doi{10.1007/JHEP01(2011)027}.

\bibitem[Hanany et~al.(2010)Hanany, Orlando, and Reffert]{HananySublattice}
Amihay Hanany, Domenico Orlando, and Susanne Reffert.
\newblock Sublattice counting and orbifolds.
\newblock \emph{J. High Energ. Phys.}, 2010:\penalty0 51, 2010.
\newblock \doi{10.1007/JHEP06(2010)051}.

\bibitem[Hart and Forcade(2008)]{HartForcade}
Gus L.~W. Hart and Rodney~W. Forcade.
\newblock Algorithm for generating derivative structures.
\newblock \emph{Phys. Rev. B}, 77:\penalty0 224115, 2008.
\newblock \doi{10.1103/PhysRevB.77.224115}.

\bibitem[Heuer and Zeiner(2010)]{A4CSL}
M.~Heuer and P.~Zeiner.
\newblock {CSL}s of the root lattice ${A}_4$.
\newblock \emph{J. Phys. Conf. Ser.}, 226:\penalty0 012024, 2010.
\newblock \doi{10.1088/1742-6596/226/1/012024}.

\bibitem[Karpenkov(2013)]{book:Karpenkov}
Oleg Karpenkov.
\newblock \emph{Geometry of Continued Fractions}.
\newblock Springer, 2013.
\newblock \doi{10.1007/978-3-642-39368-6}.

\bibitem[Martinet(2003)]{book:perfect}
Jacques Martinet.
\newblock \emph{Perfect Lattices in Euclidean Spaces}.
\newblock Springer, 2003.
\newblock \doi{10.1007/978-3-662-05167-2}.

\bibitem[Montagard and Ressayre(2009)]{Montagard}
Pierre-Louis Montagard and Nicolas Ressayre.
\newblock Regular lattice polytopes and root systems.
\newblock \emph{Bull. London Math. Soc.}, 41:\penalty0 227--241, 2009.
\newblock \doi{10.1112/blms/bdn120}.

\bibitem[OEIS()]{OEIS}
OEIS.
\newblock {The On-Line Encyclopedia of Integer Sequences}.
\newblock URL \url{https://oeis.org/}.

\bibitem[Rutherford(2009)]{Rutherford4}
John~S. Rutherford.
\newblock Sublattice enumeration. {IV}. {E}quivalence classes of plane
  sublattices by parent {P}atterson symmetry and colour lattice group type.
\newblock \emph{Acta Cryst. A}, 65:\penalty0 156--163, 2009.
\newblock \doi{10.1107/S010876730804333X}.

\end{thebibliography}

\end{document}